\documentclass[11pt,reqno]{amsart}

\usepackage[T1]{fontenc}
\usepackage{mathtools,amssymb,mathrsfs,bm}
\usepackage{graphicx}
\usepackage{enumitem}
\usepackage[a4paper,textwidth=6.5in,textheight=9in,centering]{geometry}
\usepackage{cite}
\usepackage{hyperref}
\theoremstyle{plain}
\newtheorem{theorem}{Theorem}[section]
\newtheorem{lemma}{Lemma}[section]
\newtheorem{proposition}{Proposition}[section]
\newtheorem{corollary}{Corollary}[section]
\newtheorem{question}{Question}
\theoremstyle{definition}

\newtheorem{example}{Example}[section]
\newtheorem{remark}{Remark}
\numberwithin{equation}{section}

\let\leq\leqslant
\let\geq\geqslant

\newcommand{\N}{\mathbb{N}}
\newcommand{\Zp}{\mathbb{N}_{0}}
\newcommand{\rightovernotleft}{%
  \mathrel{\vcenter{\offinterlineskip
    \halign{\hfil$##$\hfil\cr
      \rightharpoonup\cr
      \noalign{\kern-0.35ex}
      \ooalign{%
        $\leftharpoondown$\cr
        \hidewidth\raisebox{0.2ex}{\scalebox{0.4}{$~\mathbf{/}$}}\hidewidth\cr
      }\cr
    }%
  }}%
}
\newcommand{\rightoverquesmarkleft}{%
  \mathrel{\vcenter{\offinterlineskip
    \halign{\hfil$##$\hfil\cr
      \rightharpoonup\cr
      \noalign{\kern-0.35ex}
      \ooalign{%
        $\leftharpoondown$\cr
        \hidewidth\raisebox{-0.45ex}{\scalebox{0.55}{$~?$}}\hidewidth\cr
      }\cr
    }%
  }}%
}
\newcommand{\orb}{\operatorname{Orb}^{+}}

\title[Weak mixing with ${\mathscr M}_{\alpha}$-shadowing]
{Weak mixing for dynamical systems with
${\bm{\mathscr M}}_{\bm\alpha}$-shadowing}

\author[X. Wu]{Xinxing Wu}
\address{School of Mathematics and Statistics, Guizhou University of
Finance and Economics, Guiyang {\rm 550025}, China}
\email{wuxinxing5201314@163.com}

\author[J. Wei]{Jidan Wei}
\address{School of Mathematics and Statistics, Guizhou University of
Finance and Economics, Guiyang {\rm 550025}, China}
\email{jidan@mail.gufe.edu.cn}

\author[X. Zhang]{Xu Zhang}
\address{Department of Mathematics, Shandong University, Weihai,
Shandong {\rm 264209}, China}
\email{xu\_zhang\_sdu@mail.sdu.edu.cn}
\thanks{Corresponding author: Xu Zhang.}

\date{}

\hypersetup{
  pdftitle={Weak mixing for dynamical systems with M-alpha-shadowing},
  pdfauthor={Xinxing Wu, Jidan Wei, Xu Zhang}
}

\begin{document}

\begin{abstract}
This paper studies weak mixing for two classes of dynamical systems with the
\(\mathscr{M}_{\alpha}\)-shadowing property using the generic factor
characterization of weak scattering. First, for every
\(\alpha\in[0,1)\), every \(M\)-system with the
\(\mathscr{M}_{\alpha}\)-shadowing property is weakly mixing, thereby answering
\cite[Question~1]{ODH} negatively.
Second, for every \(\alpha\in(0,1)\), the
\(\mathscr{M}_{\alpha}\)-shadowing property implies weak mixing for
\(E\)-systems and, in particular, for dynamical systems whose measure center is
the whole space.
Finally, an \(E\)-system with the
\(\mathscr{M}_{0}\)-shadowing property which is not weakly mixing is constructed, showing that
the parameter range \(\alpha\in(0,1)\) in the \(E\)-system result is optimal.
\end{abstract}

\maketitle

\section{Introduction}

The shadowing property provides a fundamental mechanism for
transferring information from approximate trajectories to genuine orbits and
thereby connects orbit tracing with stability, recurrence, and complexity in
dynamical systems. The classical shadowing lemma arose from the
work of Anosov and Bowen on hyperbolic systems~\cite{Anosov,Bowen}; in its
topological formulation, the shadowing property requires every sufficiently
accurate pseudo-orbit to be uniformly traced by a genuine orbit. Walters
\cite{Walters1978} established the topological stability of
expansive homeomorphisms with the shadowing property. A systematic treatment of
shadowing theory is given by Pilyugin~\cite{Pilyugin}. For dynamical systems
with the shadowing property, Moothathu~\cite{Moothathu} showed that the minimal
points are dense in the nonwandering set and that either a nonminimal recurrent
point or a sensitive minimal subsystem implies positive topological entropy.
Li and Oprocha~\cite{LiOprocha} established the equivalence between
weak mixing and uniformly positive entropy for non-wandering systems
with the shadowing property.

Two related approaches to orbit tracing are specification-type
orbit gluing and shadowing with average error. Specification-type
notions include Bowen's specification property \textup{(Spec)}
\cite{BowenSpecification}, the \(g\)-almost product property of
Pfister and Sullivan~\cite{PfisterSullivan}, and Thompson's
almost specification property~\textup{(ASpec)}~\cite{Thompson},
the latter formulated by means of mistake functions.
Average-error tracing includes the average shadowing property
\textup{(ASP),} due to Blank~\cite{Blank}, and the asymptotic average
shadowing property \textup{(AASP),} introduced by Gu~\cite{Gu}.

Partial shadowing replaces tracing at every time by tracing on a
prescribed large set of times. Its lower-density theory began with the
\(\underline{d}\)-shadowing property of Ahmadi Dastjerdi and
Hosseini~\cite{DH} based on ergodic pseudo-orbits.
Oprocha, Ahmadi Dastjerdi and Hosseini~\cite{ODH} studied partial
shadowing of complete pseudo-orbits, while Brian, Meddaugh and Raines
\cite{BMR} developed shadowing with respect to Furstenberg families. Wu,
Oprocha and Chen~\cite{WOC} subsequently introduced, for each
\(\alpha\in[0,1)\), the family \(\mathscr{M}_{\alpha}\) of subsets
\(A\subseteq\Zp\) satisfying
\(\underline{\mathrm{Dens}}(A)>\alpha\), and the corresponding
\(\mathscr{M}_{\alpha}\)-shadowing property. If
\(0\leq\alpha_1<\alpha_2<1\), then
\(\mathscr{M}_{\alpha_2}\subseteq\mathscr{M}_{\alpha_1}\); hence
\(\mathscr{M}_{\alpha_2}\)-shadowing implies
\(\mathscr{M}_{\alpha_1}\)-shadowing. In particular,
\(\mathscr{M}_{0}\)-shadowing is the weakest member of this scale and
coincides with the \(\underline{d}\)-shadowing property. The relevant relations among these properties are
summarized by
\begin{equation}\label{eq:intro-shadowing-chain}
\mathrm{Spec} \rightovernotleft \mathrm{ASpec}
\rightovernotleft \mathrm{AASP}
\rightoverquesmarkleft \mathrm{ASP}
\rightleftharpoons
\mathscr{M}_{\alpha}\ (\forall\alpha\in[0,1))
\rightovernotleft
\mathscr{M}_{0} \rightleftharpoons
\underline{d}.
\end{equation}
In \eqref{eq:intro-shadowing-chain}, \(\rightovernotleft\) denotes a strict
implication, whereas \(\rightoverquesmarkleft\) denotes an implication whose
converse remains open. The strictness of
\(\mathrm{Spec}\rightovernotleft\mathrm{ASpec}\) is recorded in
\cite[Remark~25]{KLOPanorama}. Wu, Oprocha and Chen established
\(\mathrm{ASpec}\Rightarrow\mathrm{AASP}\) without assuming surjectivity
\cite[Corollary~6.9]{WOC}. The converse fails by combining
\cite[Theorem~18]{KLO} with the surjective system constructed in
\cite[Section~6]{KOR}, which has periodic weak specification but does not have
almost specification \cite[Theorem~6.1]{KOR}. Wu, Oprocha and Chen also proved
\(\mathrm{AASP}\Rightarrow\mathrm{ASP}\) \cite[Theorem~4.2]{WOC}; the
converse is the second question in \cite[Question~10.3]{KKO}. Moreover, ASP is
equivalent to having the \(\mathscr{M}_{\alpha}\)-shadowing property
simultaneously for all \(\alpha\in[0,1)\) \cite[Theorem~5.5]{WOC}. Since
\(\mathscr{M}_{0}\)-shadowing coincides with the
\(\underline{d}\)-shadowing property,
\cite[Theorem~5 and Example~20]{ODH} show that the implication from this
simultaneous family to \(\mathscr{M}_{0}\)-shadowing is strict.

Recently, Can and Trilles established the equivalence of AASP and Kamae's vague
specification property \cite[Theorem~4.6]{CanTrilles}, constructed nontrivial
minimal and nontrivial proximal systems with AASP
\cite[Theorems~2 and~4.6]{CanTrilles}, and showed that AASP is inherited by factors
\cite[Corollary~7.11]{CanTrilles}. These results answer
\cite[Question~10.2]{KKO} and the AASP part of \cite[Question~10.6]{KKO}.
They further proved that, for surjective
systems, AASP is equivalent to ASP together with completeness in the dynamical
Besicovitch pseudometric \cite[Theorem~4.4]{CanTrilles}. Thus, in the
surjective setting, the open implication
\(\mathrm{ASP}\Rightarrow\mathrm{AASP}\) in
\eqref{eq:intro-shadowing-chain} is equivalent to completeness in the
dynamical Besicovitch pseudometric.

Several results relate average tracing to full-support invariant
measures and measure centers. Niu~\cite{Niu} obtained weak
mixing for systems with the ASP and dense minimal
points. Kulczycki, Kwietniak and Oprocha~\cite{KKO} showed that,
for a system admitting an invariant measure of full support,
ASpec implies density of minimal points, and ASP
implies weak mixing. They further established that if
\(A\) is a closed invariant set containing the measure center, then each of
ASpec, AASP and ASP of \((A, f|_{A})\) implies the same property for \((X,f)\)
\cite[Theorems~5.1, 5.2 and~5.5]{KKO}. Wu, Oprocha and Chen extended this
transfer principle to the \(\mathscr{M}_{\alpha}\)-shadowing property
\cite[Theorem~6.2]{WOC}. They also characterized ASpec by its restriction
to the measure center \cite[Theorem~6.7]{WOC}. Oprocha and Wu~\cite{OprochaWu}
showed that \((X,f)\) has the finite average shadowing property \textup{(FinASP)}
if and only if its measure-center subsystem has
FinASP; moreover, FinASP implies weak mixing of that
subsystem. None of these results establishes weak
mixing under the assumptions \(\mathrm{supp}(X,f)=X\) and one fixed
\(\mathscr{M}_{\alpha}\)-shadowing property with
\(\alpha\in(0,1)\).

For lower-density partial shadowing, Ahmadi Dastjerdi and Hosseini~\cite{DH}
established that every minimal system with the \(\underline{d}\)-shadowing
property is weakly mixing. For chain mixing systems with the
\(\mathscr{F}_{\underline{d}}\)-shadowing property, Oprocha,
Ahmadi Dastjerdi and Hosseini~\cite{ODH} obtained weak mixing when
periodic points are dense and also in the minimal case. They
further established total transitivity under density of minimal points
\cite[Theorem~7]{ODH} and posed the following question:

\begin{question}[{\cite[Question~1]{ODH}}]\label{ques:ODH}
Does there exist a non-weakly mixing system which is chain mixing, has a dense
set of minimal points, and has the
\(\mathscr{F}_{\underline{d}}\)-shadowing property?
\end{question}

Wu et al.~\cite[Theorem~7.2]{WOC} subsequently showed that if the minimal
points are dense and \((X,f)\) has the \(\underline{d}\)-shadowing property
(and hence the \(\mathscr{F}_{\underline{d}}\)-shadowing property), then
\((X,f^{n})\) is syndetically transitive for every \(n\in\mathbb{N}\).
However, this result does not settle Question~\ref{ques:ODH}.

Under the \(\mathscr{M}_{\alpha}\)-shadowing property, density of
minimal points implies transitivity for \(\alpha\in[0,1)\), whereas the
existence of an invariant measure with full support implies transitivity for
\(\alpha\in(0,1)\). The resulting systems are therefore \(M\)-systems and
\(E\)-systems, respectively. Accordingly, we investigate weak mixing for
\(M\)-systems and \(E\)-systems with the
\(\mathscr{M}_{\alpha}\)-shadowing property. The main results and the endpoint
obstruction are summarized as follows:
\begin{align}
 M\text{-system}+\mathscr{M}_{\alpha}\text{-shadowing}
 \ (\alpha\in [0, 1))
 &\ \Rightarrow\ \text{weak mixing},
 \label{eq:intro-M-system}\\
 E\text{-system}+\mathscr{M}_{\alpha}\text{-shadowing}
 \ (\alpha\in(0, 1))
 &\ \Rightarrow\ \text{weak mixing},
 \label{eq:intro-E-system}\\
 E\text{-system}+\mathscr{M}_{0}\text{-shadowing}
 &\ \nRightarrow\ \text{weak mixing}.
 \label{eq:intro-endpoint}
\end{align}
The special case \(\alpha=0\) of \eqref{eq:intro-M-system} gives a
negative answer to Question~\ref{ques:ODH}. Together,
\eqref{eq:intro-E-system} and \eqref{eq:intro-endpoint} show that the range
\(\alpha\in(0,1)\) in \eqref{eq:intro-E-system} is sharp.
More broadly, our approach reveals that the generic factor characterization of
weak scattering provides an effective mechanism for deriving weak mixing from
partial shadowing properties.

\section{Preliminaries}\label{sec:prelim}

This section recalls the basic definitions and auxiliary results
used throughout the paper.

\subsection{Topological dynamics}

Throughout, let \(\N=\{1, 2, 3, \ldots\}\) and \(\Zp=\{0, 1, 2, \ldots\}\).
A set \(A\subseteq \Zp\) is
\begin{itemize}
  \item \textit{thick} if it contains arbitrarily long runs of
  $\Zp$, i.e., for any \(n\in \N\), there exists an \(i\in \Zp\)
  such that \(\{i, i+1, \ldots, i+n\}\subseteq A\).
  \item \textit{syndetic} if it has bounded gaps, i.e.,
  there exists \(L\in\N\) such that \([i, i+L]\cap A
  \neq \varnothing\) for any \(i\in \Zp\).
\end{itemize}

For \(A\subseteq\Zp\), the \textit{upper density} and
\textit{lower density} of \(A\) are, respectively,
\[
\overline{\mathrm{Dens}}(A)=\limsup_{n\to\infty}
 \frac{\#(A\cap\{0,\ldots,n\})}{n+1}
 \ \text{ and }\
\underline{\mathrm{Dens}}(A)=\liminf_{n\to\infty}
 \frac{\#(A\cap\{0,\ldots,n\})}{n+1}.
\]
If the limit exists, the \textit{density} of \(A\) is defined by
\[
 \mathrm{Dens}(A)=\lim_{n\to\infty}
 \frac{\#(A\cap\{0,\ldots,n\})}{n+1}.
\]
Fix any \(\alpha\in[0, 1)\), and denote by $\mathscr{M}_{\alpha}$
(resp. $\mathscr{M}^{\alpha}$) the family consisting of sets
$A\subseteq \Zp$ with $\underline{\mathrm{Dens}}(A)>\alpha$ (resp.
$\overline{\mathrm{Dens}}(A)>\alpha$).

A \emph{dynamical system} is a pair \((X,f)\) consisting of a compact metric
space \((X,d)\) and a continuous surjection \(f\colon X\longrightarrow X\).
For \(x\in X\) and \(A,B\subseteq X\),
we define the \textit{return-time sets}
\(N_{f}(x,B)=\{n \in \N: f^{n}(x)\in B\}\)
and \(N_{f}(A,B)=\{n\in\N:f^{n}(A)\cap B\neq\varnothing\}.\)

The \textit{orbit} of a point \(x\in X\) under \(f\) is the set
\(\orb(x,f)=\{f^{n}(x):n\in \Zp\}.\) A point $x\in X$ is
a \textit{transitive point} of \((X,f)\) if \(\overline{\orb(x,f)}=X\).
Denote by \(\mathrm{Tran}(X,f)\) the set of all
transitive points of \((X,f)\).
A pair $(x, y)$ is \emph{proximal} if
$\liminf_{n\to \infty}d(f^{n}(x), f^{n}(y))=0$.

A dynamical system \((X,f)\) is \textit{transitive} if
\(N_{f}(U,V)\neq\varnothing\) for any non-empty open sets
\(U\), \(V\subseteq X\); \textit{weakly mixing} if
the product system $(X\times X, f\times f)$ is transitive;
\textit{mixing} if \(\Zp\setminus N_{f}(U,V)\)
is finite for any non-empty open sets \(U\), \(V\subseteq X\);
and \textit{minimal} if \(\mathrm{Tran}(X,f)=X\).
Equivalently, $(X, f)$ is minimal if and only if it
contains no proper subsystems.

A point \(x\in X\) is called a \textit{minimal point} if
  \((\overline{\orb(x,f)},f|_{\overline{\orb(x,f)}})\) is minimal.
It is well known that
\begin{itemize}
  \item a dynamical system $(X, f)$ is weakly mixing if and only if \(N_{f}(U, V)\)
  is thick for any nonempty open subsets \(U,V\subseteq X\).
  \item a point \(x\in X\) is a minimal point if and only if \(N_{f}(x,U)\) is syndetic
  for any neighborhood \(U\) of \(x\).
\end{itemize}

A dynamical system $(X, f)$ is
\begin{itemize}
  \item an \emph{\(M\)-system} if it is transitive and the set of all minimal points is dense in $X$;
  \item \emph{equicontinuous} if, for any \(\varepsilon>0\), there exists \(\delta>0\) such that
for all $x, y\in X$ with \(d(x,y)<\delta\) and all $n\in \Zp$, we have
\(d(f^{n}(x),f^{n}(y))<\varepsilon\).
\end{itemize}

Two dynamical systems are \emph{weakly disjoint} if their product
system is transitive. A transitive system is \emph{weakly scattering} if it
is weakly disjoint from every minimal equicontinuous system.

We recall the notion of generic factors which was introduced by
Huang and Ye~\cite{HY}. Let \((X,f)\) and \((Y,g)\) be
two transitive systems. If there exists a continuous map
\(\pi\colon \mathrm{Tran}(X,f)\longrightarrow \mathrm{Tran}(Y,g)\) with
\(\pi(f(x))=g(\pi(x))\) for all \(x\in \mathrm{Tran}(X,f)\), then we say
\(\pi\) is a \emph{generic homomorphism} from \((X,f)\) to \((Y,g)\), \((Y,g)\) is
a \emph{generic factor} of \((X,f)\), and \((X,f)\) is a \emph{generic extension} of
\((Y,g)\).  It is not hard to see that if \((X,f)\) is minimal and
\(\pi\colon (X,f)\longrightarrow (Y,g)\) is a generic homomorphism, then
\(\pi\) is a factor map.

The following results of Huang and Ye describe weak scattering in terms of
equicontinuous generic factors and, for \(E\)-systems, weak mixing.

\begin{lemma}[{\cite[Theorem~3.8]{HY}}]\label{lem:HY-factor}
Let \((X,f)\) be a transitive system.  Then \((X,f)\) is weakly scattering if
and only if it has no non-trivial equicontinuous generic factor.
\end{lemma}

\begin{lemma}[{\cite[Proposition~4.14~(2)]{HY}}]\label{lem:HY-E}
Let \((X,f)\) be an \(E\)-system. Then \((X,f)\)
is weakly scattering if and only if it is weakly mixing.
\end{lemma}

\subsection{Probability measures}
Following Walters~\cite{Walters}, we recall the basic notions and
results concerning invariant measures used below.
The \(\sigma\)-algebra of Borel
subsets of \(X\) is denoted by \(\mathscr{B}(X)\). We denote by
\(\mathcal{M}(X)\) the collection of all probability measures
defined on \((X,\mathscr{B}(X))\); its members are called
\emph{Borel probability measures} on \(X\). Each \(x\in X\) determines a
\textit{Dirac point measure} \(\delta_{x}\in\mathcal{M}(X)\)
defined by
\[
 \delta_{x}(A)=
 \begin{cases}
  1,&x\in A,\\
  0,&x\notin A,
 \end{cases}
 \quad (A\in\mathscr{B}(X)).
\]

The weak-$*$ topology on \(\mathcal{M}(X)\) is the smallest topology making
each of the maps
\[
 \mu\longmapsto\int_{X}\varphi\,d\mu
 \quad(\varphi\in C(X))
\]
continuous. Thus \(\mu_{m}\overset{w^{*}}{\longrightarrow}\mu\) in
\(\mathcal{M}(X)\) if and only if
\[
 \lim_{m\to\infty}\int_{X}\varphi\,d\mu_{m}
 =\int_{X}\varphi\,d\mu
 \quad\text{for every }\varphi\in C(X).
\]
It is well known that \(\mathcal{M}(X)\) is a compact
metrizable space in the weak-$*$ topology.

A measure \(\mu\in\mathcal{M}(X)\) is \emph{\(f\)-invariant} if
\(\mu(f^{-1}(A))=\mu(A)\) for every \(A\in\mathscr{B}(X)\). Denote by
\(\mathcal{M}(X,f)\) the set of all \(f\)-invariant Borel probability
measures on \(X\); that is,
\[
 \mathcal{M}(X,f)
 =\{\mu\in\mathcal{M}(X):\mu(f^{-1}(A))=\mu(A)
       \text{ for every }A\in\mathscr{B}(X)\}.
\]
Equivalently, by \cite[Theorem~6.8]{Walters},
\(\mu\in\mathcal{M}(X,f)\) if and only if
\[
 \int_{X}\varphi\circ f\,d\mu=\int_{X}\varphi\,d\mu
 \quad\text{for every }\varphi\in C(X).
\]

A measure \(\mu\in\mathcal{M}(X,f)\) is \emph{ergodic} if
\(\mu(A)\in\{0,1\}\) whenever \(A\in\mathscr{B}(X)\) and
\(f^{-1}(A)=A\). Equivalently, the same conclusion holds whenever
\(\mu(f^{-1}(A)\mathbin{\triangle}A)=0\).

The \emph{support} of \(\mu\in\mathcal{M}(X)\) is the closed set
\[
 \operatorname{supp}(\mu)
 =\{x\in X:\mu(U)>0\text{ for every open neighborhood }U\text{ of }x\}.
\]
The \emph{measure center} of \((X,f)\) is
\[
 \mathrm{supp}(X,f)
 =\overline{\bigcup_{\mu\in\mathcal{M}(X,f)}
                    \operatorname{supp}(\mu)}.
\]

A transitive system \((X,f)\) is called an \emph{\(E\)-system}
if \(\mathcal{M}(X,f)\) contains a measure with full support.
Since \(X\) has a countable base, \(\mathrm{supp}(X,f)=X\) if and only if
\(\mathcal{M}(X,f)\) contains a measure with full support.

A point \(x\in X\) is \emph{generic for} \(\mu\in\mathcal{M}(X,f)\) if
\[
 \frac{1}{m}\sum_{n=0}^{m-1}\delta_{f^{n}(x)}
 \overset{w^{*}}{\longrightarrow}\mu
 \quad\text{as }m\longrightarrow\infty.
\]
The following lemma combines Lemma~6.13, Theorem~6.14, and the remarks
following Theorems~6.4 and~6.10 of \cite{Walters}.

\begin{lemma}[{\cite{Walters}}]\label{lem:Walters-measure-facts}
Let $(X, f)$ be a dynamical system.
Then the following statements hold:
\begin{enumerate}[label=\textup{(\arabic*)},ref=\textup{(\arabic*)}]
\item\label{item:Walters-generic-point}
If \(\mu\in\mathcal{M}(X,f)\) is ergodic, then \(\mu\)-almost every
point is generic for \(\mu\).
\item\label{item:Walters-ergodic-decomposition}
If \(U\subseteq X\) is open and \(\mu\in\mathcal{M}(X,f)\) satisfies
\(\mu(U)>0\), then there exists an ergodic measure
\(\nu\in\mathcal{M}(X,f)\) such that \(\nu(U)>0\).
\item\label{item:Walters-continuity-set}
If \(\mu_{m},\mu\in\mathcal{M}(X)\),
\(\mu_{m}\overset{w^{*}}{\longrightarrow}\mu\), and
\(A\in\mathscr{B}(X)\) satisfies \(\mu(\partial A)=0\), then
\[
 \lim_{m\to\infty}\mu_{m}(A)=\mu(A).
\]
\end{enumerate}
\end{lemma}

\begin{lemma}[{\cite[Proposition~8.2.8~(ii)]{Bogachev}}]
\label{lem:zero-boundary-neighborhood}
Let $x\in X$ and $U$ be an open neighborhood of $x$. Then,
for every \(\mu\in\mathcal{M}(X)\), there exists an open set \(V\) such that
\(x\in V\subseteq\overline{V}\subseteq U\) and \(\mu(\partial V)=0\).
\end{lemma}

\subsection{Shadowing properties}

Following \cite{ODH,WOC}, we recall the notions and results
concerning pseudo-orbits and shadowing properties used below.

Let \(\delta>0\). Following \cite[Definition~2.1]{WOC}, a
sequence \(\{x_{i}\}_{i=0}^{\infty}\subseteq X\) is called
\begin{enumerate}[label=\textup{(\arabic*)}]
  \item a \emph{\(\delta\)-pseudo-orbit} if
\[
d(f(x_{i}),x_{i+1})<\delta \quad (\forall i\in \Zp).
\]
  \item a \emph{\(\delta\)-ergodic pseudo-orbit} if
\[
\lim_{n\to\infty}
\frac{\#\{0\leq i<n:d(f(x_{i}),x_{i+1})\geq\delta\}}{n}=0.
\]
  \item a \emph{\(\delta\)-average pseudo-orbit} if there exists
  \(N\in\N\) such that
\[
 \frac{1}{n}\sum_{i=0}^{n-1}
 d(f(x_{i+k}),x_{i+k+1})<\delta
 \quad(\forall n\geq N,\ \forall k\in\Zp).
\]
\end{enumerate}

A finite \(\delta\)-pseudo-orbit \(\{x_{i}\}_{i=0}^{n}\) is called a
\emph{\(\delta\)-chain} from \(x_{0}\) to \(x_{n}\) of length \(n\).

A dynamical system \((X,f)\) is
\begin{itemize}
  \item \emph{chain transitive} if, for any $\delta>0$ and any
  $x, y\in X$, there exists a $\delta$-chain from $x$ to $y$;
  \item \emph{chain mixing} if, for any \(\delta>0\) and any
\(x,y\in X\), there exists \(N\in\mathbb{N}\) such that for any
\(n>N\), there exists a \(\delta\)-chain of length \(n\) from
\(x\) to \(y\).
\end{itemize}

Richeson and Wiseman~\cite{RW} proved that $(X, f)$ is chain mixing if
and only if $(X, f^{n})$ is chain transitive for every $n\in \N$.
The following lemma gives a uniform version of chain mixing in which the
lower bound on the chain length is independent of the endpoints.

\begin{lemma}[{\cite[Lemma~1]{ODH}}]\label{lem:uniform-chain}
Assume that \((X,f)\) is chain mixing. Then for any \(\varepsilon>0\), there
exists \(N=N(\varepsilon)\in \N\) such that for any \(x,y\in X\) and any
\(n\geq N\), there exist points \(\xi_{1},\ldots,\xi_{n}\) such that the sequence
\[
 (x,\xi_{1},\ldots,\xi_{n},y)
\]
is an \(\varepsilon\)-chain.
\end{lemma}

A system \((X,f)\) has
\begin{itemize}
  \item the \emph{average shadowing property} \textup{(ASP)} if, for any
  \(\varepsilon>0\), there exists \(\delta>0\) such that every
  \(\delta\)-average pseudo-orbit \(\{x_{i}\}_{i=0}^{\infty}\) is
  \(\varepsilon\)-shadowed on average by some point \(z\in X\), i.e.,
\[
 \limsup_{n\to\infty}\frac{1}{n}\sum_{i=0}^{n-1}
 d(f^{i}(z),x_{i})<\varepsilon.
\]
  \item the \emph{\(\mathscr{F}_{\underline{d}}\)-shadowing property} if,
  for any \(\varepsilon>0\), there exists \(\delta>0\) such that every
\(\delta\)-pseudo-orbit \(\{x_{i}\}_{i=0}^{\infty}\) is
$\mathscr{F}_{\underline{d}}$-$\varepsilon$-shadowed by some point \(z\in X\), i.e.,
\[
 \underline{\mathrm{Dens}}(\{j\in\Zp:d(f^{j}(z),x_{j})<\varepsilon\})>0.
\]
  \item the \emph{\(\underline{d}\)-shadowing property} if, for any
  \(\varepsilon>0\), there exists \(\delta>0\) such that every
  \(\delta\)-ergodic pseudo-orbit \(\{x_{i}\}_{i=0}^{\infty}\) is
  $\underline{d}$-$\varepsilon$-shadowed by some point \(z\in X\),
  i.e.,
\[
 \underline{\mathrm{Dens}}(\{j\in\Zp:d(f^{j}(z),x_{j})<\varepsilon\})>0.
\]
  \item the \emph{\(\mathscr{M}_{\alpha}\)-shadowing property} if, for any
\(\varepsilon>0\), there exists \(\delta>0\) such that every \(\delta\)-ergodic
pseudo-orbit \(\{x_{i}\}_{i=0}^{\infty}\) is
\(\mathscr{M}_{\alpha}\)-\(\varepsilon\)-shadowed by some point
\(z\in X\), i.e.,
\[
\{j\in\Zp:d(f^{j}(z),x_{j})<\varepsilon\}\in \mathscr{M}_{\alpha}.
\]
\end{itemize}

By \cite[Theorem~5.5]{WOC}, ASP is equivalent to the simultaneous validity
of the \(\mathscr{M}_{\alpha}\)-shadowing property for all
\(\alpha\in[0,1)\). By definition, the \(\mathscr{M}_{0}\)-shadowing property
is precisely the \(\underline{d}\)-shadowing property. The following
lemma shows that, for chain mixing systems, the latter is equivalent to the
\(\mathscr{F}_{\underline{d}}\)-shadowing property.

\begin{lemma}[{\cite[Lemma~4]{ODH}}]\label{lem:ODH-equivalence}
Assume that \((X,f)\) is chain mixing. The following conditions are
equivalent:
\begin{enumerate}[label=\textup{(\arabic*)}]
 \item \((X,f)\) has the \(\underline{d}\)-shadowing property;
 \item \((X,f)\) has the \(\mathscr{F}_{\underline{d}}\)-shadowing property.
\end{enumerate}
\end{lemma}

\section{Weak mixing for \(M\)-systems with
\(\mathscr{M}_{\alpha}\)-shadowing}\label{sec:factor}

Using the generic factor characterization of weak scattering, we
prove in this section that, for every \(\alpha\in[0,1)\), every \(M\)-system
with the \(\mathscr{M}_{\alpha}\)-shadowing property is weakly mixing. As a
corollary, we answer Question~\ref{ques:ODH} negatively.

\begin{lemma}\label{prop:no-factor}
Suppose that \((X,f)\) has a dense set of minimal points. If
\((X, f)\) has the \(\mathscr{M}_{0}\)-shadowing property, then it
has no nontrivial equicontinuous generic factor.
\end{lemma}

\begin{proof}
Since the \(\mathscr{M}_{0}\)-shadowing property is the
\(\underline{d}\)-shadowing property, \cite[Theorem~7.2]{WOC} shows that
\((X,f)\) is transitive.
Suppose, to the contrary, that \((X, f)\) admits a nontrivial
equicontinuous generic factor \((Y,g)\) with generic homomorphism
\(\pi\colon \mathrm{Tran}(X,f)\longrightarrow \mathrm{Tran}(Y,g)\).
Since every point of \(Y\) is an equicontinuity point,
\cite[Theorem~2.4]{AAB} gives \(\mathrm{Tran}(Y,g)=Y\),
and hence \((Y,g)\) is minimal and equicontinuous.
By \cite[Proposition~2.14]{Sun2020}, there exists a compatible metric
\(\rho\) on \(Y\) for which \(g\) is an isometry.  We use this metric throughout
the proof.

Since \(Y\) is nontrivial, put \(D=\operatorname{diam}(Y)>0\).
Let \(p_{X}\colon X\times Y\longrightarrow X\) be the first-coordinate projection, and let
\[
 \Gamma=\overline{\{(x,\pi(x)):x\in\mathrm{Tran}(X,f)\}}\subseteq X\times Y,
 \quad
 \Gamma(x)=\{y\in Y:(x,y)\in\Gamma\}.
\]
It can be verified that
\begin{itemize}
  \item $\Gamma$ is compact and $p_{X}(\Gamma)=X$;
  \item Since $f\times g$ is continuous, we have $(f\times g)(\Gamma)\subseteq
  \overline{(f\times g)(\{(x,\pi(x)):x\in\mathrm{Tran}(X,f)\})}
  =\overline{\{(f(x), g(\pi(x))): x\in\mathrm{Tran}(X,f)\}}
  =\overline{\{(f(x), \pi(f(x))): x\in\mathrm{Tran}(X,f)\}}\subseteq\Gamma$, i.e., \(\Gamma\) is
positively invariant under \(f\times g\);
  \item For any \(x\in\mathrm{Tran}(X,f)\), it is clear that \(\pi(x)\in\Gamma(x)\).
  Conversely, if \(y\in\Gamma(x)\), then there exists a
sequence \(\{x_{n}\}_{n=1}^{\infty}\subseteq\mathrm{Tran}(X,f)\) such that
\((x_{n},\pi(x_{n}))\longrightarrow (x,y)\), implying
\(\pi(x)=\lim_{n\to \infty}\pi(x_{n})=y\). Therefore, \(\Gamma(x)=\{\pi(x)\}\).
\end{itemize}

Fix a point \(x_{0}\in\mathrm{Tran}(X,f)\) and
set \(\eta=\frac{D}{100}>0\).
Let
\[
 O=X\setminus p_{X}(\Gamma\cap
 (X\times(Y\setminus B(\pi(x_{0}),\eta)))).
\]
Clearly, \(O\) is an open neighborhood of \(x_{0}\) by
\(\Gamma(x_0)=\{\pi(x_0)\}\).
Moreover, let \(u\in O\) and \(y\in\Gamma(u)\).  If
\(y\notin B(\pi(x_{0}),\eta)\), then
\((u,y)\in\Gamma\cap(X\times(Y\setminus B(\pi(x_{0}),\eta)))\),
which implies
\[
 u\in p_{X}(\Gamma\cap
 (X\times(Y\setminus B(\pi(x_{0}),\eta)))),
\]
contrary to \(u\in O\).  Therefore,
\[
 \Gamma(u)\subseteq B(\pi(x_{0}),\eta)\quad (\forall u\in O).
\]
Since the set of minimal points is dense in \(X\), choose a minimal point
\(x_{1}\in O\). This, together with \(p_{X}(\Gamma)=X\) and
\(\Gamma(x_{1})\subseteq B(\pi(x_{0}),\eta)\), implies
\(\Gamma(x_{1})\neq\varnothing\) and
\(\operatorname{diam}(\Gamma(x_{1}))<2\eta\). Fix a point
\(y_{0}\in\Gamma(x_{1})\). Since
\(\overline{\operatorname{Orb}(y_{0},g)}=Y\),
there exists \(r_{0}\in \N\) such that
\begin{equation}\label{eq:separated}
 \rho(y_{0},g^{r_{0}}(y_{0}))>\frac{D}{3}.
\end{equation}

Furthermore, let
\[
 U=X\setminus p_{X}(\Gamma\cap
 (X\times(Y\setminus
 \bigcup_{y\in\Gamma(x_{1})}B(y,\eta)))).
\]
Clearly, \(U\) is an open neighborhood of \(x_{1}\) by
\(\Gamma(x_{1})\subseteq
\bigcup_{y\in\Gamma(x_{1})}B(y,\eta)\).
Moreover, let \(u\in U\) and \(y'\in\Gamma(u)\). If
\(y'\notin\bigcup_{y\in\Gamma(x_{1})}B(y,\eta)\), then
\((u,y')\in\Gamma\cap
(X\times(Y\setminus\bigcup_{y\in\Gamma(x_{1})}B(y,\eta)))\), which implies
\[
 u\in p_{X}(\Gamma\cap
 (X\times(Y\setminus\bigcup_{y\in\Gamma(x_{1})}B(y,\eta)))),
\]
contrary to \(u\in U\). Therefore,
\begin{equation}\label{eq:Gamma-U}
 \Gamma(u)\subseteq\bigcup_{y\in\Gamma(x_{1})}B(y,\eta)
 \quad(\forall u\in U).
\end{equation}
Choose an open neighborhood \(V\) of \(x_{1}\) with \(\overline{V}\subseteq U\).
Since \(x_{1}\) is minimal, there exists \(K\in \N\) such that every interval
of \(K+1\) consecutive iterates of \(x_{1}\) meets \(V\).
Since \(\overline{V}\) is compact and contained in \(U\), one has
\(\operatorname{dist}(\overline{V},X\setminus U)>0\);
and since \(f\) is uniformly continuous, there exists
\(\varepsilon>0\) such that
\begin{equation}\label{eq:propagate}
 d(x,x')<\varepsilon
 \Longrightarrow
 d(f^{k}(x),f^{k}(x'))<\frac{1}{3}\operatorname{dist}(\overline{V},X\setminus U)
 \quad (\forall x,x'\in X,\ \forall k=0,1,\ldots,K).
\end{equation}
By the \(\mathscr{M}_{0}\)-shadowing property, there exists
\(\delta>0\) such that every \(\delta\)-ergodic pseudo-orbit is
\(\mathscr{M}_{0}\)-\(\varepsilon\)-shadowed by some point of \(X\).
Put \(M_{0}=0\) and define inductively
\[
 M_{n+1}=M_{n}+2^{M_{n}+K+1}\quad(n\in\Zp).
\]
Then \(\lim_{n\to\infty}\frac{M_{n}}{M_{n+1}}=0\). Define
\(\xi=\{\xi_{i}\}_{i=0}^{\infty}\) by
\[
\xi_{i}=
\begin{cases}
 f^{i}(x_{1}),
 &i\in\displaystyle\bigcup_{n\in\Zp}[M_{2n},M_{2n+1}),\\
 f^{i+r_{0}}(x_{1}),
 &i\in\displaystyle\bigcup_{n\in\Zp}[M_{2n+1},M_{2n+2}),
\end{cases}
\]
i.e.,
\[
\begin{aligned}
 \xi={}&x_{1},f(x_{1}),\ldots,f^{M_{1}-1}(x_{1}),
 f^{M_{1}+r_{0}}(x_{1}),\ldots,f^{M_{2}-1+r_{0}}(x_{1}),\\
 &f^{M_{2}}(x_{1}),\ldots,f^{M_{3}-1}(x_{1}),
 f^{M_{3}+r_{0}}(x_{1}),\ldots,f^{M_{4}-1+r_{0}}(x_{1}),\ldots.
\end{aligned}
\]
Since \(\{M_{n}-1:n\in\N\}\) has density zero, \(\xi\) is a
\(\delta\)-ergodic pseudo-orbit.
By the \(\mathscr{M}_{0}\)-shadowing property, there exists
\(z\in X\) such that \(\xi\) is
\(\mathscr{M}_{0}\)-\(\varepsilon\)-shadowed by \(z\). Put
\(\Lambda=\{i\in\mathbb{N}:d(f^{i}(z),\xi_{i})<\varepsilon\}\).
Then \(\underline{\mathrm{Dens}}(\Lambda)>0\).
Let
\[
 J_{0}=\bigcup_{n\in \Zp}[M_{2n},M_{2n+1}-K)
 \ \text{ and }\
 J_{1}=\bigcup_{n\in \Zp}[M_{2n+1},M_{2n+2}-K).
\]
The condition
\(\lim_{n\to \infty}\frac{M_{n}}{M_{n+1}}=0\) implies that
\(\overline{\mathrm{Dens}}(J_{0})=\overline{\mathrm{Dens}}(J_{1})=1\).
This, together with \(\underline{\mathrm{Dens}}(\Lambda)>0\) and
\cite[Lemma~1.1]{DH}, implies that both
\(\Lambda\cap J_{0}\) and \(\Lambda\cap J_{1}\) are infinite.
Therefore, there exist indices
\(0<i<j\), with \(i\in\Lambda\cap J_{0}\) and
\(j\in\Lambda\cap J_{1}\), where shadowing takes place; that is,
\[
 d(f^{i}(z),f^{i}(x_{1}))<\varepsilon
 \text{ and }
 d(f^{j}(z),f^{j+r_{0}}(x_{1}))<\varepsilon.
\]
Together with \eqref{eq:propagate}, these inequalities imply
\[
 \begin{aligned}
 d(f^{i+k}(z),f^{i+k}(x_{1}))
 &<\frac{1}{3}\operatorname{dist}(\overline{V},X\setminus U),\\
 d(f^{j+k}(z),f^{j+r_{0}+k}(x_{1}))
 &<\frac{1}{3}\operatorname{dist}(\overline{V},X\setminus U),
 \end{aligned}
 \quad(\forall k=0,1,\ldots,K).
\]
By the choice of \(K\), there exist \(0\leq s,t\leq K\)
such that
\[
 f^{i+s}(x_{1})\in V
 \text{ and }
 f^{j+r_{0}+t}(x_{1})\in V.
\]
Therefore,
\[
f^{i+s}(z),\ f^{j+t}(z)\in U.
\]
Since the transitive points are dense in \(X\), choose
\(z'\in\mathrm{Tran}(X,f)\) sufficiently close to \(z\) that
\[
 f^{i+s}(z'),\ f^{j+t}(z')\in U.
\]
The identities
\(\Gamma(f^{i+s}(z'))
 =\{\pi(f^{i+s}(z'))\}=\{g^{i+s}(\pi(z'))\}\)
 and \(\Gamma(f^{j+t}(z'))
 =\{\pi(f^{j+t}(z'))\}=\{g^{j+t}(\pi(z'))\}\)
(by \(z'\in\mathrm{Tran}(X,f)\))
together with~\eqref{eq:Gamma-U}, give
\begin{equation}\label{eq:g-1}
 g^{i+s}(\pi(z')),\ g^{j+t}(\pi(z'))\in
 \bigcup_{y\in\Gamma(x_{1})}B(y,\eta).
\end{equation}
The positive invariance of \(\Gamma\) and
\((x_{1},y_{0})\in\Gamma\) yield
\[
 g^{i+s}(y_{0})\in\Gamma(f^{i+s}(x_{1}))
 \ \text{ and }\
 g^{j+r_{0}+t}(y_{0})\in\Gamma(f^{j+r_{0}+t}(x_{1})).
\]
As \(f^{i+s}(x_{1}),f^{j+r_{0}+t}(x_{1})\in V\subseteq U\),
another application of~\eqref{eq:Gamma-U} gives
\[
 \Gamma(f^{i+s}(x_{1}))\cup\Gamma(f^{j+r_{0}+t}(x_{1}))
 \subseteq\bigcup_{y\in\Gamma(x_{1})}B(y,\eta),
\]
implying
\begin{equation}\label{eq:g-2}
 g^{i+s}(y_{0}),\ g^{j+r_{0}+t}(y_{0})
 \in\bigcup_{y\in\Gamma(x_{1})}B(y,\eta).
\end{equation}
Using \eqref{eq:g-1}, \eqref{eq:g-2}, the isometry of \(g\),
and \(\operatorname{diam}(\Gamma(x_{1}))<2\eta\), we obtain
\[
 \rho(\pi(z'),y_{0})=\rho(g^{i+s}(\pi(z')), g^{i+s}(y_{0}))
 <2\eta+\operatorname{diam}(\Gamma(x_{1}))<4\eta,
\]
and
\[
 \rho(\pi(z'),g^{r_{0}}(y_{0}))
 =\rho(g^{j+t}(\pi(z')), g^{j+t+r_{0}}(y_{0}))
 <2\eta+\operatorname{diam}(\Gamma(x_{1}))<4\eta.
\]
Consequently,
\[
 \rho(y_{0},g^{r_{0}}(y_{0}))\leq \rho(\pi(z'),y_{0})
 +\rho(\pi(z'),g^{r_{0}}(y_{0}))
 <8\eta=\frac{2D}{25}<\frac{D}{3},
\]
contrary to \eqref{eq:separated}.
Hence $(X, f)$ has no nontrivial equicontinuous generic
factor.
\end{proof}

\begin{theorem}\label{thm:main}
Let \(\alpha\in[0, 1)\). If \((X,f)\) is an \(M\)-system with the
\(\mathscr{M}_{\alpha}\)-shadowing property, then it is weakly mixing.
\end{theorem}

\begin{proof}
By \eqref{eq:intro-shadowing-chain}, \((X,f)\) has the
\(\mathscr{M}_{0}\)-shadowing property. Then, Lemma~\ref{prop:no-factor} together
with Lemma~\ref{lem:HY-factor} shows that \((X,f)\) is weakly scattering.
Since every \(M\)-system is an \(E\)-system
\cite[Theorem~3.8~\textup{(1)}]{DayanGlasner2015},
Lemma~\ref{lem:HY-E} implies that \((X,f)\) is weakly mixing.
\end{proof}

\begin{corollary}\label{cor:ODH-question}
Suppose that \((X,f)\) is a chain mixing system with a dense set of minimal points.
If \((X,f)\) has the \(\mathscr{F}_{\underline{d}}\)-shadowing property,
then it is weakly mixing.
\end{corollary}

\begin{proof}
By \cite[Theorem~7]{ODH}, \((X,f)\) is transitive. Since its
minimal points are dense, \((X,f)\) is an \(M\)-system. Moreover,
Lemma~\ref{lem:ODH-equivalence} together with
\eqref{eq:intro-shadowing-chain} shows that \((X,f)\) has the
\(\mathscr{M}_{0}\)-shadowing property. The conclusion follows directly from
Theorem~\ref{thm:main} with \(\alpha=0\).
\end{proof}

\begin{remark}
Corollary~\ref{cor:ODH-question} gives a direct negative answer to
Question~\ref{ques:ODH}.
\end{remark}

\section{Weak mixing for \(E\)-systems with \(\mathscr{M}_{\alpha}\)-shadowing}

In contrast to the result for \(M\)-systems in
Section~\ref{sec:factor}, the corresponding result for \(E\)-systems requires
\(\alpha>0\). More precisely, we prove in this section that, for every
\(\alpha\in(0,1)\), every \(E\)-system
with the \(\mathscr{M}_{\alpha}\)-shadowing property is weakly mixing. We also
construct an \(E\)-system with the \(\mathscr{M}_{0}\)-shadowing property that
is not weakly mixing, showing that the weak-mixing conclusion fails at
\(\alpha=0\).

\begin{lemma}\label{lem:measure-return}
Suppose that the measure center of \((X,f)\) is \(X\).
For every nonempty open set \(U\) of \(X\), there exist an ergodic measure
\(\nu\in\mathcal{M}(X,f)\) and a \(\nu\)-generic point
\(x\in U\cap\operatorname{supp}(\nu)\). Moreover, if
\(V\) is an open neighborhood of \(x\) with \(\nu(\partial V)=0\), then
the density
\(\mathrm{Dens}(\bigcup_{0\leq s\leq K}((N_{f}(x,V)-s)\cap\N))\)
exists for every \(K\in\mathbb{N}\), and
\(\lim_{K\to\infty}\mathrm{Dens}(\bigcup_{0\leq s\leq K}
 ((N_{f}(x,V)-s)\cap\N))=1.\)
\end{lemma}

\begin{proof}
There exists \(\mu\in\mathcal{M}(X,f)\) such that \(\mu(U)>0\)
by \(\mathrm{supp}(X,f)=X\). By
Lemma~\ref{lem:Walters-measure-facts}~%
\ref{item:Walters-ergodic-decomposition}, there exists an ergodic measure
\(\nu\in\mathcal{M}(X,f)\) with \(\nu(U)>0\). Since the set of
\(\nu\)-generic points has full \(\nu\)-measure by
Lemma~\ref{lem:Walters-measure-facts}~%
\ref{item:Walters-generic-point}, and
\(\operatorname{supp}(\nu)\) also has full \(\nu\)-measure, one can choose a
\(\nu\)-generic point
\(x\in U\cap\operatorname{supp}(\nu)\).

For any \(K\in\mathbb{N}\), put \(C_{K}=\bigcup_{s=0}^{K}f^{-s}(V).\)
Since
\(\partial C_{K}
 \subseteq\bigcup_{s=0}^{K}\partial(f^{-s}(V))
 \subseteq\bigcup_{s=0}^{K}f^{-s}(\partial V),\)
the invariance of \(\nu\) gives
\(\nu(\partial C_{K})
 \leq\sum_{s=0}^{K}\nu(f^{-s}(\partial V))
 =(K+1)\nu(\partial V)=0.\)
Moreover,
\[
 \bigcup_{0\leq s\leq K}((N_{f}(x,V)-s)\cap\N)
 =\{n\in\N:f^{n}(x)\in C_{K}\}.
\]
Since \(x\) is \(\nu\)-generic,
\(\frac{1}{m}\sum_{n=0}^{m-1}\delta_{f^{n}(x)}
 \overset{w^{*}}{\longrightarrow}\nu
 \quad\text{as }m\longrightarrow\infty.\)
This, together with \(\nu(\partial C_{K})=0\)
and Lemma~\ref{lem:Walters-measure-facts}~%
\ref{item:Walters-continuity-set}, implies
\(\lim_{m\to\infty}\frac{1}{m}
 \#\{0\leq n<m:f^{n}(x)\in C_{K}\}
 =\nu(C_{K}).\)
By the preceding set equality, this yields
\[
 \mathrm{Dens}\Bigg(\bigcup_{0\leq s\leq K}
 ((N_{f}(x,V)-s)\cap\N)\Bigg)=\nu(C_{K}).
\]
Let \(C_{\infty}=\bigcup_{K\in\mathbb{N}}C_{K}\). Since
\(x\in\operatorname{supp}(\nu)\) and \(x\in
V\subseteq C_{\infty}\), one has \(\nu(C_{\infty})
\geq \nu(V)>0\). Moreover,
\(f^{-1}(C_{\infty})\subseteq C_{\infty}\), while invariance gives
\(\nu(f^{-1}(C_{\infty}))=\nu(C_{\infty})\). Hence ergodicity
of \(\nu\) yields
\(\nu(C_{\infty})=1\). Finally, since \(C_{K}\) increases to
\(C_{\infty}\), we have
\(\lim_{K\to\infty}\mathrm{Dens}(\bigcup_{0\leq s\leq K}
 ((N_{f}(x,V)-s)\cap\N))
 =\lim_{K\to\infty}\nu(C_{K})=\nu(C_{\infty})=1.\)
\end{proof}

\begin{proposition}\label{prop:center-transitive}
Let \(\alpha\in(0,1)\). If the measure center of \((X,f)\) is \(X\) and
\((X,f)\) has the \(\mathscr{M}_{\alpha}\)-shadowing property, then
\((X,f)\) is transitive.
\end{proposition}

\begin{proof}
Let \(U_{0},U_{1}\subseteq X\) be nonempty open sets. By
Lemma~\ref{lem:measure-return}, for \(r=0,1\), there exist an ergodic measure
\(\nu_{r}\) and a \(\nu_{r}\)-generic point
\(x_{r}\in U_{r}\cap\operatorname{supp}(\nu_{r})\). By
Lemma~\ref{lem:zero-boundary-neighborhood}, for \(r=0,1\), there exists an
open neighborhood \(V_{r}\) of \(x_{r}\) such that
\(\overline{V_{r}}\subseteq U_{r}\) and \(\nu_{r}(\partial V_{r})=0\).
By Lemma~\ref{lem:measure-return}, there exists \(K\in\N\) such that
\begin{equation}\label{eq:return-density}
 \mathrm{Dens}\Bigg(\bigcup_{0\leq s\leq K}
 ((N_{f}(x_{r},V_{r})-s)\cap\N)\Bigg)>1-\alpha
 \quad(r=0,1).
\end{equation}
Since \(\overline{V_{r}}\) is compact and contained in \(U_{r}\), one has
\(\operatorname{dist}(\overline{V_{r}},X\setminus U_{r})>0\) for \(r=0,1\).
Since \(f\) is uniformly continuous, there exists \(\varepsilon>0\) such that,
for $r=0, 1$,
\begin{equation}\label{eq:transitive-propagate}
 d(x,x')<\varepsilon
 \Longrightarrow
 d(f^{k}(x),f^{k}(x'))<
 \frac{1}{3}\operatorname{dist}(\overline{V_{r}},X\setminus U_{r})
 \quad(\forall x,x'\in X,\ \forall k=0,1,\ldots,K).
\end{equation}
By the \(\mathscr{M}_{\alpha}\)-shadowing property, there exists
\(\delta>0\) such that every \(\delta\)-ergodic pseudo-orbit is
\(\mathscr{M}_{\alpha}\)-\(\varepsilon\)-shadowed by some point of \(X\).
Put \(M_{0}=0\) and define inductively
\[
 M_{n+1}=M_{n}+2^{M_{n}+K+1}\quad(n\in\Zp).
\]
Then \(\lim_{n\to\infty}\frac{M_{n}}{M_{n+1}}=0\). Define
\(\xi=\{\xi_{i}\}_{i=0}^{\infty}\) by
\[
 \xi_{i}=
 \begin{cases}
  f^{i}(x_{0}),
  &i\in\displaystyle\bigcup_{n\in\Zp}[M_{2n},M_{2n+1}),\\
  f^{i}(x_{1}),
 &i\in\displaystyle\bigcup_{n\in\Zp}[M_{2n+1},M_{2n+2}).
 \end{cases}
\]
i.e.,
\[
\begin{aligned}
 \xi={}&x_{0},f(x_{0}),\ldots,f^{M_{1}-1}(x_{0}),
 f^{M_{1}}(x_{1}),\ldots,f^{M_{2}-1}(x_{1}),\\
 &f^{M_{2}}(x_{0}),\ldots,f^{M_{3}-1}(x_{0}),
 f^{M_{3}}(x_{1}),\ldots,f^{M_{4}-1}(x_{1}),\ldots.
\end{aligned}
\]
Since \(\{M_{n}-1:n\in\N\}\) has density zero, \(\xi\) is a
\(\delta\)-ergodic pseudo-orbit.
By the \(\mathscr{M}_{\alpha}\)-shadowing property, there exists
\(z\in X\) such that \(\xi\) is
\(\mathscr{M}_{\alpha}\)-\(\varepsilon\)-shadowed by \(z\). Put
\(\Lambda=\{i\in\mathbb{N}:d(f^{i}(z),\xi_{i})<\varepsilon\}\).
Then \(\underline{\mathrm{Dens}}(\Lambda)>\alpha\).
Let
\[
 J_{0}=\bigcup_{n\in \Zp}[M_{2n},M_{2n+1}-K)
 \text{ and }
 J_{1}=\bigcup_{n\in \Zp}[M_{2n+1},M_{2n+2}-K).
\]
The condition
\(\lim_{n\to\infty}\frac{M_{n}}{M_{n+1}}=0\) implies that
\(\overline{\mathrm{Dens}}(J_{0})=\overline{\mathrm{Dens}}(J_{1})=1\). By
\eqref{eq:return-density}, for \(r=0,1\),
\[
\begin{aligned}
 &\underline{\mathrm{Dens}}\Bigg(\Lambda\cap
   \bigcup_{0\leq s\leq K}
   ((N_{f}(x_{r},V_{r})-s)\cap\N)\Bigg)\\
 &\quad\geq \underline{\mathrm{Dens}}(\Lambda)
   +\mathrm{Dens}\Bigg(\bigcup_{0\leq s\leq K}
   ((N_{f}(x_{r},V_{r})-s)\cap\N)\Bigg)-1>0.
\end{aligned}
\]
This, together with \(\overline{\mathrm{Dens}}(J_{0})=\overline{\mathrm{Dens}}(J_{1})=1\) and
\cite[Lemma~1.1]{DH}, implies that
both \(\Lambda\cap J_{0}\cap\bigcup_{0\leq s\leq K}
((N_{f}(x_{0},V_{0})-s)\cap\N)\) and
\(\Lambda\cap J_{1}\cap\bigcup_{0\leq s\leq K}
((N_{f}(x_{1},V_{1})-s)\cap\N)\) are infinite.
Therefore, there exist indices \(0<i<j\) satisfying \(j-i>K\), with
\(i\in\Lambda\cap J_{0}\cap\bigcup_{0\leq s\leq K}
((N_{f}(x_{0},V_{0})-s)\cap\N)\) and
\(j\in\Lambda\cap J_{1}\cap\bigcup_{0\leq s\leq K}
((N_{f}(x_{1},V_{1})-s)\cap\N)\), where shadowing takes place; that is,
\[
 d(f^{i}(z),f^{i}(x_{0}))<\varepsilon
 \ \text{ and }\
 d(f^{j}(z),f^{j}(x_{1}))<\varepsilon.
\]
Together with \eqref{eq:transitive-propagate}, these inequalities imply
\[
 \begin{aligned}
 d(f^{i+k}(z),f^{i+k}(x_{0}))
 &<\frac{1}{3}\operatorname{dist}(\overline{V_{0}},X\setminus U_{0}),\\
 d(f^{j+k}(z),f^{j+k}(x_{1}))
 &<\frac{1}{3}\operatorname{dist}(\overline{V_{1}},X\setminus U_{1}),
 \end{aligned}
 \quad(\forall k=0,1,\ldots,K).
\]
By the choice of \(i\) and \(j\), there exist \(0\leq s,t\leq K\) such that
\(f^{i+s}(x_{0})\in V_{0}\) and \(f^{j+t}(x_{1})\in V_{1}\).
Therefore,
\[
 f^{i+s}(z)\in U_{0}
 \ \text{ and }\
 f^{j+t}(z)\in U_{1}.
\]
Since \(j-i>K\) and \(0\leq s,t\leq K\), one has
\((j+t)-(i+s)\in\N\), and thus,
\[
 f^{j+t}(z)
 =f^{(j+t)-(i+s)}(f^{i+s}(z))
 \in f^{(j+t)-(i+s)}(U_{0})\cap U_{1}
 \neq\varnothing.
\]
This proves that \((X,f)\) is transitive.
\end{proof}

The proof of the following lemma is analogous to that of
Lemma~\ref{prop:no-factor}; we include the details for completeness.

\begin{lemma}\label{lem:center-no-factor}
Let \(\alpha\in(0,1)\). Suppose that the measure center of \((X,f)\) is
\(X\). If \((X,f)\) has the \(\mathscr{M}_{\alpha}\)-shadowing property,
then it has no nontrivial equicontinuous generic factor.
\end{lemma}

\begin{proof}
By Proposition~\ref{prop:center-transitive}, \((X,f)\) is transitive.
Suppose, to the contrary, that \((X,f)\) admits a nontrivial
equicontinuous generic factor \((Y,g)\) with generic homomorphism
\(\pi\colon\mathrm{Tran}(X,f)\longrightarrow\mathrm{Tran}(Y,g)\).
Since every point of \(Y\) is an equicontinuity point,
\cite[Theorem~2.4]{AAB} gives \(\mathrm{Tran}(Y,g)=Y\),
and hence \((Y,g)\) is minimal and equicontinuous.
By \cite[Proposition~2.14]{Sun2020}, there exists a compatible metric
\(\rho\) on \(Y\) for which \(g\) is an isometry. We use this metric throughout
the proof.

Since \(Y\) is nontrivial, put \(D=\operatorname{diam}(Y)>0\).
Let \(p_{X}\colon X\times Y\longrightarrow X\) be the first-coordinate projection, and let
\[
 \Gamma=\overline{\{(x,\pi(x)):x\in\mathrm{Tran}(X,f)\}}\subseteq X\times Y,
 \quad
 \Gamma(x)=\{y\in Y:(x,y)\in\Gamma\}.
\]
As in the proof of Lemma~\ref{prop:no-factor}, the relation
\(\Gamma\) has the following properties:
\begin{itemize}
  \item \(\Gamma\) is compact and \(p_{X}(\Gamma)=X\);
  \item \(\Gamma\) is positively invariant under \(f\times g\);
  \item For any \(x\in\mathrm{Tran}(X,f)\),
  \(\Gamma(x)=\{\pi(x)\}\).
\end{itemize}

Fix a point \(x_{0}\in\mathrm{Tran}(X,f)\) and
set \(\eta=\frac{D}{100}>0\). Let
\[
 O=X\setminus p_{X}(\Gamma\cap
 (X\times(Y\setminus B(\pi(x_{0}),\eta)))).
\]
Clearly, \(O\) is an open neighborhood of \(x_{0}\) by
\(\Gamma(x_{0})=\{\pi(x_{0})\}\).
Moreover, let \(u\in O\) and \(y\in\Gamma(u)\). If
\(y\notin B(\pi(x_{0}),\eta)\), then
\((u,y)\in\Gamma\cap(X\times(Y\setminus B(\pi(x_{0}),\eta)))\), which implies
\[
 u\in p_{X}(\Gamma\cap
 (X\times(Y\setminus B(\pi(x_{0}),\eta)))),
\]
contrary to \(u\in O\). Therefore,
\[
 \Gamma(u)\subseteq B(\pi(x_{0}),\eta)\quad(\forall u\in O).
\]
By Lemma~\ref{lem:measure-return}, there exist an ergodic measure
\(\nu\in\mathcal{M}(X,f)\) and a \(\nu\)-generic point
\(x_{1}\in O\cap\operatorname{supp}(\nu)\). This, together with
\(p_{X}(\Gamma)=X\) and
\(\Gamma(x_{1})\subseteq B(\pi(x_{0}),\eta)\), implies
\(\Gamma(x_{1})\neq\varnothing\) and
\(\operatorname{diam}(\Gamma(x_{1}))<2\eta\). Fix a point
\(y_{0}\in\Gamma(x_{1})\). Since
\(\overline{\operatorname{Orb}(y_{0},g)}=Y\), there exists
\(r_{0}\in\N\) such that
\begin{equation}\label{eq:center-separated}
 \rho(y_{0},g^{r_{0}}(y_{0}))>\frac{D}{3}.
\end{equation}

Furthermore, let
\[
 U=X\setminus p_{X}(\Gamma\cap
 (X\times(Y\setminus
 \bigcup_{y\in\Gamma(x_{1})}B(y,\eta)))).
\]
Clearly, \(U\) is an open neighborhood of \(x_{1}\) by
\(\Gamma(x_{1})\subseteq
\bigcup_{y\in\Gamma(x_{1})}B(y,\eta)\).
Moreover, let \(u\in U\) and \(y'\in\Gamma(u)\). If
\(y'\notin\bigcup_{y\in\Gamma(x_{1})}B(y,\eta)\), then
\((u,y')\in\Gamma\cap
(X\times(Y\setminus\bigcup_{y\in\Gamma(x_{1})}B(y,\eta)))\), which implies
\[
 u\in p_{X}(\Gamma\cap
 (X\times(Y\setminus\bigcup_{y\in\Gamma(x_{1})}B(y,\eta)))),
\]
contrary to \(u\in U\). Therefore,
\begin{equation}\label{eq:center-Gamma-U}
 \Gamma(u)\subseteq\bigcup_{y\in\Gamma(x_{1})}B(y,\eta)
 \quad(\forall u\in U).
\end{equation}
By Lemma~\ref{lem:zero-boundary-neighborhood}, there exists an open
neighborhood \(V\) of \(x_{1}\) such that
\(\overline{V}\subseteq U\) and \(\nu(\partial V)=0\). For
\(K\in\N\), put $R_{K}^{0}=\bigcup_{0\leq s\leq K}((N_{f}(x_{1}, V)-s)\cap \N)$
and
$R_{K}^{1}=\bigcup_{0\leq s\leq K}((N_{f}(x_{1}, V)-r_0-s)\cap \N)$.
Applying Lemma~\ref{lem:measure-return} yields
\[
 \lim_{K\to\infty}\mathrm{Dens}(R_{K}^{0})
 =\lim_{K\to\infty}\mathrm{Dens}(R_{K}^{1})=1,
\]
implying that there exists \(K\in\N\) such that
\begin{equation}\label{eq:two-return-densities}
 \mathrm{Dens}(R_{K}^{0})>1-\frac{\alpha}{2}
 \quad\text{and}\quad
 \mathrm{Dens}(R_{K}^{1})>1-\frac{\alpha}{2}.
\end{equation}
Since \(\overline{V}\) is compact and contained in \(U\), one has
\(\operatorname{dist}(\overline{V},X\setminus U)>0\); and
since \(f\) is uniformly continuous, there exists
\(\varepsilon>0\) such that
\begin{equation}\label{eq:center-propagate}
 d(x,x')<\varepsilon
 \Longrightarrow
 d(f^{k}(x),f^{k}(x'))<\frac{1}{3}
 \operatorname{dist}(\overline{V},X\setminus U)
 \quad(\forall x,x'\in X,\ \forall k=0,1,\ldots,K).
\end{equation}
By the \(\mathscr{M}_{\alpha}\)-shadowing property, there exists
\(\delta>0\) such that every \(\delta\)-ergodic pseudo-orbit is
\(\mathscr{M}_{\alpha}\)-\(\varepsilon\)-shadowed by some point of \(X\).
Put \(M_{0}=0\) and define inductively
\[
 M_{n+1}=M_{n}+2^{M_{n}+K+1}\quad(n\in\Zp).
\]
Then \(\lim_{n\to\infty}\frac{M_{n}}{M_{n+1}}=0\). Define
\(\xi=\{\xi_{i}\}_{i=0}^{\infty}\) by
\[
\xi_{i}=
\begin{cases}
 f^{i}(x_{1}),
 &i\in\displaystyle\bigcup_{n\in\Zp}[M_{2n},M_{2n+1}),\\
 f^{i+r_{0}}(x_{1}),
 &i\in\displaystyle\bigcup_{n\in\Zp}[M_{2n+1},M_{2n+2}).
\end{cases}
\]
Since \(\{M_{n}-1:n\in\N\}\) has density zero, \(\xi\) is a
\(\delta\)-ergodic pseudo-orbit.

By the \(\mathscr{M}_{\alpha}\)-shadowing property, there exists
\(z\in X\) such that \(\xi\) is
\(\mathscr{M}_{\alpha}\)-\(\varepsilon\)-shadowed by \(z\). Put
\[
 \Lambda=\{i\in\mathbb{N}:d(f^{i}(z),\xi_{i})<\varepsilon\}.
\]
Then \(\underline{\mathrm{Dens}}(\Lambda)>\alpha\).
Let
\[
 J_{0}=\bigcup_{n\in\Zp}[M_{2n},M_{2n+1}-K)
 \ \text{ and }\
 J_{1}=\bigcup_{n\in\Zp}[M_{2n+1},M_{2n+2}-K).
\]
The condition
\(\lim_{n\to\infty}\frac{M_{n}}{M_{n+1}}=0\) implies that
\(\overline{\mathrm{Dens}}(J_{0})=\overline{\mathrm{Dens}}(J_{1})=1\). By
\eqref{eq:two-return-densities},

\[
\begin{aligned}
 \underline{\mathrm{Dens}}(\Lambda\cap R_{K}^{0})
 &\geq \underline{\mathrm{Dens}}(\Lambda)
   +\mathrm{Dens}(R_{K}^{0})-1
   > \frac{\alpha}{2}>0,\\
 \underline{\mathrm{Dens}}(\Lambda\cap R_{K}^{1})
 &\geq \underline{\mathrm{Dens}}(\Lambda)
   +\mathrm{Dens}(R_{K}^{1})-1
   > \frac{\alpha}{2}>0.
\end{aligned}
\]
These, together with
\(\overline{\mathrm{Dens}}(J_{0})=\overline{\mathrm{Dens}}(J_{1})=1\) and
\cite[Lemma~1.1]{DH}, imply that both
\(\Lambda\cap R_{K}^{0}\cap J_{0}\) and
\(\Lambda\cap R_{K}^{1}\cap J_{1}\) are infinite.
Therefore, there exist indices \(0<i<j\), with
\(i\in\Lambda\cap R_{K}^{0}\cap J_{0}\) and
\(j\in\Lambda\cap R_{K}^{1}\cap J_{1}\), where shadowing takes place; that is,
\[
 d(f^{i}(z),f^{i}(x_{1}))<\varepsilon
 \ \text{ and }\
 d(f^{j}(z),f^{j+r_{0}}(x_{1}))<\varepsilon.
\]
Together with \eqref{eq:center-propagate}, these inequalities imply
\[
 \begin{aligned}
 d(f^{i+k}(z),f^{i+k}(x_{1}))
 &<\frac{1}{3}\operatorname{dist}(\overline{V},X\setminus U),\\
 d(f^{j+k}(z),f^{j+r_{0}+k}(x_{1}))
 &<\frac{1}{3}\operatorname{dist}(\overline{V},X\setminus U),
 \end{aligned}
 \quad(\forall k=0,1,\ldots,K).
\]
By the definitions of \(R_{K}^{0}\) and \(R_{K}^{1}\), there exist
\(0\leq s,t\leq K\) such that
\[
 f^{i+s}(x_{1})\in V
 \ \text{ and }\
 f^{j+r_{0}+t}(x_{1})\in V.
\]
Therefore,
\[
 f^{i+s}(z),\ f^{j+t}(z)\in U.
\]
Since the transitive points are dense in \(X\), choose
\(z'\in\mathrm{Tran}(X,f)\) sufficiently close to \(z\) that
\[
 f^{i+s}(z'),\ f^{j+t}(z')\in U.
\]
The identities
\(\Gamma(f^{i+s}(z'))
=\{\pi(f^{i+s}(z'))\}=\{g^{i+s}(\pi(z'))\}\)
and \(\Gamma(f^{j+t}(z'))
=\{\pi(f^{j+t}(z'))\}=\{g^{j+t}(\pi(z'))\}\)
(by \(z'\in\mathrm{Tran}(X,f)\)) together with
\eqref{eq:center-Gamma-U} give
\begin{equation}\label{eq:center-g-1}
 g^{i+s}(\pi(z')),\ g^{j+t}(\pi(z'))\in
 \bigcup_{y\in\Gamma(x_{1})}B(y,\eta).
\end{equation}
The positive invariance of \(\Gamma\) and
\((x_{1},y_{0})\in\Gamma\) yield
\[
 g^{i+s}(y_{0})\in\Gamma(f^{i+s}(x_{1}))
 \text{ and }
 g^{j+r_{0}+t}(y_{0})\in\Gamma(f^{j+r_{0}+t}(x_{1})).
\]
As \(f^{i+s}(x_{1}),f^{j+r_{0}+t}(x_{1})\in V\subseteq U\), another
application of \eqref{eq:center-Gamma-U} gives
\[
 \Gamma(f^{i+s}(x_{1}))\cup\Gamma(f^{j+r_{0}+t}(x_{1}))
 \subseteq\bigcup_{y\in\Gamma(x_{1})}B(y,\eta),
\]
implying
\begin{equation}\label{eq:center-g-2}
 g^{i+s}(y_{0}),\ g^{j+r_{0}+t}(y_{0})
 \in\bigcup_{y\in\Gamma(x_{1})}B(y,\eta).
\end{equation}
Using \eqref{eq:center-g-1}, \eqref{eq:center-g-2}, the isometry of \(g\),
and \(\operatorname{diam}(\Gamma(x_{1}))<2\eta\), we obtain
\[
 \rho(\pi(z'),y_{0})
 =\rho(g^{i+s}(\pi(z')),g^{i+s}(y_{0}))
 <2\eta+\operatorname{diam}(\Gamma(x_{1}))<4\eta,
\]
and
\[
 \rho(\pi(z'),g^{r_{0}}(y_{0}))
 =\rho(g^{j+t}(\pi(z')),g^{j+t+r_{0}}(y_{0}))
 <2\eta+\operatorname{diam}(\Gamma(x_{1}))<4\eta.
\]
Consequently,
\[
 \rho(y_{0},g^{r_{0}}(y_{0}))
 \leq\rho(\pi(z'),y_{0})+\rho(\pi(z'),g^{r_{0}}(y_{0}))
 <8\eta=\frac{2D}{25}<\frac{D}{3},
\]
contrary to \eqref{eq:center-separated}.
Hence \((X,f)\) has no nontrivial equicontinuous generic factor.
\end{proof}

\begin{theorem}\label{thm:E-system}
Let \(\alpha\in(0,1)\). If \((X,f)\) is an \(E\)-system with the
\(\mathscr{M}_{\alpha}\)-shadowing property, then it is weakly mixing.
\end{theorem}

\begin{proof}
Since \((X,f)\) is an \(E\)-system,
\(\mathrm{supp}(X,f)=X\). Lemma~\ref{lem:center-no-factor} together with
Lemma~\ref{lem:HY-factor} shows that \((X,f)\) is weakly scattering. Since
\((X,f)\) is an \(E\)-system, Lemma~\ref{lem:HY-E} then shows that it is weakly
mixing.
\end{proof}

\begin{corollary}\label{cor:center}
Let \(\alpha\in(0,1)\). Suppose that the measure center of
\((X,f)\) is \(X\). If \((X,f)\) has the
\(\mathscr{M}_{\alpha}\)-shadowing property, then it is weakly mixing.
\end{corollary}

\begin{proof}
By Proposition~\ref{prop:center-transitive}, \((X,f)\) is
transitive. This, together with \(\mathrm{supp}(X,f)=X\), implies that
\((X,f)\) is an \(E\)-system. The conclusion follows directly from
Theorem~\ref{thm:E-system}.
\end{proof}

\begin{lemma}\label{lem:proximal-fixed-point}
Suppose that \(p\) is a fixed point of \((X,f)\) and that \((x,p)\) is proximal
for every \(x\in X\). Then \((X,f)\) has the
\(\mathscr{M}_{0}\)-shadowing property.
\end{lemma}

\begin{proof}
Fix any \(\varepsilon>0\). For any
\(x\in X\), since \((x,p)\) is proximal,
there exists \(j\in\N\)
such that
\(d(f^{j}(x),p)<\frac{\varepsilon}{2}.\)
Consequently,
\(\left\{f^{-j}(B(p,\frac{\varepsilon}{2})):
 j\in\N\right\}\)
is an open cover of \(X\). By compactness, there exist
\(j_{1},\ldots,j_{m}\in\N\) such that
\[
 X=\bigcup_{r=1}^{m}f^{-j_{r}}
 \left(B\left(p,\frac{\varepsilon}{2}\right)\right).
\]
Put \(K=\max\{j_{1},\ldots,j_{m}\}\). Then, for any \(x\in X\),
\(d(f^{j}(x),p)<\frac{\varepsilon}{2}\) for some
\(j\in\{0,1,\ldots,K\}\).
By uniform continuity of $f$, there exists \(\delta>0\) such
that every finite \(\delta\)-chain
\((x_{0},x_{1},\ldots,x_{K})\) satisfies
\[
 d(f^{j}(x_{0}),x_{j})<\frac{\varepsilon}{2}
 \quad (\forall j=0,1,\ldots,K).
\]

For any \(\delta\)-ergodic pseudo-orbit \(\xi=\{x_{i}\}_{i=0}^{\infty}\),
put
\(B=\{i\in\mathbb{N}_{0}:d(f(x_{i}),x_{i+1})\geq\delta\}.\)
Then \(\overline{\mathrm{Dens}}(B)=0\). Set \(L=K+1\) and let
\(Q=\{q\in\mathbb{N}_{0}:[qL,qL+K)\cap B=\varnothing\}.\)
For any \(q\in Q\), the finite sequence
\((x_{qL},x_{qL+1},\ldots,x_{qL+K})\) is a \(\delta\)-chain. By the
definition of \(K\), there exists \(j_{q}\in\{0,1,\ldots,K\}\) such that
\[
 d(f^{j_{q}}(x_{qL}),p)<\frac{\varepsilon}{2}.
\]
It follows from the choice of \(\delta\) that
\[
 d(f^{j_{q}}(x_{qL}),x_{qL+j_{q}})<\frac{\varepsilon}{2}.
\]
Therefore,
\[
 d(x_{qL+j_{q}},p)
 \leq d(x_{qL+j_{q}},f^{j_{q}}(x_{qL}))
       +d(f^{j_{q}}(x_{qL}),p)
 <\frac{\varepsilon}{2}+\frac{\varepsilon}{2}
 =\varepsilon.
\]
Put \(A=\{qL+j_{q}:q\in Q\}\). For any \(N\in\mathbb{N}\), let
\(m=\left\lfloor\frac{N}{L}\right\rfloor\). Every
\(q\in\{0,1,\ldots,m-1\}\setminus Q\) satisfies
\([qL,qL+K)\cap B\neq\varnothing\). Since these intervals are pairwise
disjoint,
\[
 \#(\{0,1,\ldots,m-1\}\setminus Q)
 \leq\#(B\cap[0,mL)).
\]
On the other hand, every \(q\in Q\cap\{0,1,\ldots,m-1\}\) determines a
distinct point \(qL+j_{q}\in[0,N)\). Therefore,
\[
\begin{aligned}
 \frac{\#(A\cap[0,N))}{N}
 &\geq\frac{\#(Q\cap\{0,1,\ldots,m-1\})}{N}\\
 &=\frac{m}{N}-\frac{\#(\{0,1,\ldots,m-1\}\setminus Q)}{N}\\
 &\geq\frac{m}{N}-\frac{\#(B\cap[0,mL))}{N}.
\end{aligned}
\]
Since \(\frac{m}{N}\longrightarrow\frac{1}{L}\) and
\(0\leq\frac{\#(B\cap [0, mL))}{N}
\leq\frac{\#(B\cap [0, N))}{N}\longrightarrow 0\), it
follows that \(\underline{\mathrm{Dens}}(A)\geq\frac{1}{L}>0\). For any
\(q\in Q\), the fixed-point property of \(p\) gives
\(d(f^{qL+j_{q}}(p),x_{qL+j_{q}})=d(p,x_{qL+j_{q}})<\varepsilon\). Thus \(p\)
\(\varepsilon\)-shadows \(\xi\) along \(A\), and hence \((X,f)\) has the
\(\mathscr{M}_{0}\)-shadowing property.
\end{proof}

\begin{example}\label{ex:center-endpoint}
Let \(Y\subseteq\{0,1\}^{\mathbb{N}_{0}}\) be the nontrivial proximal subshift
constructed in \cite[Section~6]{KLO}, and put \(g=\left.\sigma\right|_{Y}\),
where \(\sigma\) denotes the shift map.
Then the following properties hold:
\begin{enumerate}[label=\textup{(\arabic*)}]
\item By the construction in \cite[Section~6]{KLO},
\(0^{\infty}\in Y\), and hence \(g(0^{\infty})=0^{\infty}\). Since
\((Y,g)\) is proximal, every point of \(Y\) is proximal to
\(0^{\infty}\);
\item by \cite[Theorem~27]{KLO}, \((Y,g)\) is mixing and there
exists an invariant Borel probability measure \(\mu\) on \(Y\) such that
\(\operatorname{supp}(\mu)=Y\).
\end{enumerate}

Let \(X\) be the quotient space of the disjoint union of two
copies of \(Y\) obtained by identifying their copies of \(0^{\infty}\) to a
single point. More precisely, let
\[
 X=(Y\times\{0,1\})/{\sim},
\]
where \(\sim\) is the equivalence relation that identifies
\((0^{\infty},0)\) with \((0^{\infty},1)\) and leaves all other points in
singleton equivalence classes, and
let \(q\colon Y\times\{0,1\}\longrightarrow X\) be the quotient map. The equivalence
relation is closed, so \(X\) is compact and metrizable. Fix a compatible
metric \(d\) on \(X\) and put
\(\mathbf{p}_{X}=q(0^{\infty},0)=q(0^{\infty},1)
=\{(0^{\infty},0),(0^{\infty},1)\}\). Define
\(f\colon(X,d)\longrightarrow(X,d)\) and
\(\nu\colon\mathscr{B}(X)\longrightarrow[0,1]\), respectively, by
\[
 f(q(y,e))=q(g(y),1-e)
\quad(y\in Y,\ e\in\{0,1\}),
\]
and
\[
\nu(A)=\frac{1}{2}\bigl(\mu(\{y\in Y:q(y,0)\in A\})
+\mu(\{y\in Y:q(y,1)\in A\})\bigr)
\quad(A\in\mathscr{B}(X)).
\]
It can be verified that the following properties hold:
\begin{enumerate}[label=\textup{(\roman*)}]
\item Since \(g(0^{\infty})=0^{\infty}\), the map \(f\) is
well-defined and continuous, and \(\mathbf{p}_{X}\) is a fixed point of \(f\).
Moreover,
\[
 f^{2}(q(y,0))=q(g^{2}(y),0)
 \quad\text{and}\quad
 f^{2}(q(y,1))=q(g^{2}(y),1)
 \quad(y\in Y).
\]
Since \(g\) is mixing, \(g^{2}\) is surjective; hence the displayed identities
show that \(f^{2}\), and consequently \(f\), is surjective;
\item Given any \(x\in X\),
since \(q\) is surjective, there exist \(y\in Y\) and
\(e\in\{0,1\}\) such that \(x=q(y,e)\). Since \(y\) is proximal to
\(0^{\infty}\), there exists an increasing sequence
\((n_{k})\) in \(\mathbb{N}_{0}\) such that
\(g^{n_{k}}(y)\longrightarrow0^{\infty}\). Since \(\{0,1\}\) is finite, there exist
\(e'\in\{0,1\}\) and a
subsequence \((n'_{k})\) of \((n_{k})\) such that
\((e+n'_{k})\bmod 2=e'\) for every \(k\). By the continuity
of \(q\),
\(f^{n'_{k}}(x)=q(g^{n'_{k}}(y),e')
 \longrightarrow q(0^{\infty},e')=\mathbf{p}_{X},\)
and thus \(\liminf_{n\to\infty}d(f^{n}(x),\mathbf{p}_{X})=0\).
Lemma~\ref{lem:proximal-fixed-point} therefore shows that \((X,f)\) has the
\(\mathscr{M}_{0}\)-shadowing property;
\item For every \(A\in\mathscr{B}(X)\), the \(g\)-invariance of
\(\mu\) gives
\[
\begin{aligned}
 \nu(f^{-1}(A))
 &=\frac{1}{2}(
   \mu(\{y\in Y:q(g(y),1)\in A\})
   +\mu(\{y\in Y:q(g(y),0)\in A\}))\\
 &=\frac{1}{2}(
   \mu(\{y\in Y:q(y,0)\in A\})
   +\mu(\{y\in Y:q(y,1)\in A\}))
 =\nu(A).
\end{aligned}
\]
Hence
\(\nu\in\mathcal{M}(X,f)\). Moreover, \(\operatorname{supp}(\mu)=Y\) gives
\(X\supseteq\operatorname{supp}(\nu)
 \supseteq q(Y\times\{0\})\cup q(Y\times\{1\})=X.\)
In particular, \(\mathrm{supp}(X,f)=X\);
\item The system \((X,f)\) is transitive. Indeed, since \((Y,g)\)
is nontrivial and mixing, \(0^{\infty}\) is not isolated in \(Y\), and hence
\(\mathbf{p}_{X}\) has empty interior in \(X\). Let \(U,V\subseteq X\) be
nonempty open sets. There exist \(e_{0},e_{1}\in\{0,1\}\) and nonempty open
sets \(U',V'\subseteq Y\setminus\{0^{\infty}\}\) such that
\[
 q(U'\times\{e_{0}\})\subseteq U
 \quad\text{and}\quad
 q(V'\times\{e_{1}\})\subseteq V.
\]
Since \(g\) is mixing, there exists \(N\in\mathbb{N}_{0}\) such that
\(g^{n}(U')\cap V'\neq\varnothing\) for every \(n\geq N\). Taking
\(n_{0}\geq N\) with \(n_{0}\equiv e_{1}-e_{0}\pmod 2\), we obtain
\[
 f^{n_{0}}(q(U'\times\{e_{0}\}))
 \cap q(V'\times\{e_{1}\})\neq\varnothing.
\]
Thus \(f^{n_{0}}(U)\cap V\neq\varnothing\);
\item The system \((X,f)\) is not weakly mixing. Let
\(W=Y\setminus\{0^{\infty}\}\), and put
\(U_{0}=q(W\times\{0\})\) and \(U_{1}=q(W\times\{1\})\). Then \(U_{0}\) and \(U_{1}\) are
nonempty open subsets of \(X\). If \(n\) is even, then
\(f^{n}(U_{0})\cap U_{1}=\varnothing\), whereas if \(n\) is odd, then
\(f^{n}(U_{0})\cap U_{0}=\varnothing\). Consequently,
\[
 (f\times f)^{n}(U_{0}\times U_{0})
 \cap(U_{0}\times U_{1})=\varnothing
 \quad(n\in\mathbb{N}_{0}),
\]
so \(f\times f\) is not transitive, i.e., $f$ is not weakly mixing.
\end{enumerate}

\smallskip

It follows from \textup{(i)}--\textup{(v)} that
\((X,f)\) is an \(E\)-system with the
\(\mathscr{M}_{0}\)-shadowing property which is not weakly mixing.
\end{example}

\begin{remark}
Theorem~\ref{thm:main} includes the endpoint \(\alpha=0\): every
\(M\)-system with the \(\mathscr{M}_{0}\)-shadowing property is weakly mixing.
In contrast, Example~\ref{ex:center-endpoint} gives an \(E\)-system with the
\(\mathscr{M}_{0}\)-shadowing property which is not weakly mixing. Hence the
range \(\alpha\in(0,1)\) in Theorem~\ref{thm:E-system} is sharp.
\end{remark}

\end{document}